\documentclass{amsart}

\usepackage{amsmath}
\usepackage{amscd}
\usepackage{amsfonts}
\usepackage{amssymb}
\usepackage{mathrsfs}
\usepackage{latexsym}
\usepackage[all]{xy}
\usepackage{multirow}
\usepackage[active]{srcltx}
\usepackage{array}
\usepackage{color}

\newtheorem{thm}{Theorem}[section]
\newtheorem{THM}{Theorem}

\newtheorem{cor}[thm]{Corollary}
\newtheorem{prop}[thm]{Proposition}

\newtheorem{lemma}[thm]{Lemma}

\newtheorem{remark}[thm]{Remark}

\newcommand{\vecs}{\mathfrak{X}}
\newcommand{\dbyd}[2]{{\frac{\partial #1}{\partial #2}}}

\newcommand{\lieg}{\mathfrak{g}}

\newcommand{\rar}{\rightarrow}

\newcommand{\lm}[1]{\mathbb{#1}}

\DeclareMathOperator{\rank}{rank}
\DeclareMathOperator{\Pic}{Pic}
\DeclareMathOperator{\sing}{sing}
\DeclareMathOperator{\D}{D}
\DeclareMathOperator{\Aff}{Aff}

\DeclareMathOperator{\Image}{Im}

\title[Diagonal Poisson structures]{A characterization of diagonal Poisson structures}
\author[R. Lima and J.V. Pereira]{Renan Lima and  Jorge Vit\'{o}rio Pereira}

\begin{document}

\maketitle

\begin{abstract}
The degeneracy locus  of a generically symplectic Poisson structure on a Fano manifold is always a singular hypersurface.
We prove that there exists just one family of generically symplectic Poisson structures in Fano manifold
with cyclic Picard group having a reduced simple normal crossing degeneracy locus.
\end{abstract}

\tableofcontents
\section{Introduction}

\subsection{Poisson structures}
Let $X$ be a complex manifold {of dimension $n$}. A Poisson structure on $X$ is a bivector field $\Pi \in H^0(X,\bigwedge^ 2 TX)$ such that the Schouten bracket $[\Pi , \Pi] \in H^0(X, \bigwedge^3 TX)$ vanishes identically. The vanishing of the Schouten bracket implies that the image of the morphism
\begin{align*}
\Omega^1_X &\longrightarrow TX \\
\eta &\longmapsto   i_\eta \Pi
\end{align*}
is an involutive subsheaf of $TX$, and the induced foliation is called the symplectic foliation of $\Pi$. The most basic invariant attached to $\Pi$ is its rank  which
is the generic rank of this involutive subsheaf of  $TX$.
Thanks to the  anti-symmetry  of $\Pi$, its rank is an even integer $2r$ where $r$ is the largest integer such that $\Pi^r$ does not vanish identically.
When $n$ is even and the rank is equal to $n$, we say that the Poisson structure is generically symplectic.  In this case, $\{ \Pi^{  {n/2}}=0 \}$ defines
a divisor which we call the degeneracy divisor of $\Pi$. 

In this paper we are mainly interested on generically symplectic Poisson structures which have reduced 
and simple normal crossing degeneracy divisor. In the terminology of Goto \cite{Goto}, these are log-symplectic 
manifolds (complex manifolds endowed with a closed and non-degenerate  logarithmic $2$-form). We avoid Goto's  terminology
in this paper because in the context of  $C^{\infty}$ Poisson geometry, a log-symplectic manifold is usually defined as a  generically symplectic
Poisson manifold with reduced and smooth degeneracy divisor, see \cite{Cavalcanti,GualtieriLi}. Log symplectic  manifolds in this strict sense are also 
 labeled topologically stable Poisson manifolds \cite{Radko},  b-Poisson manifolds \cite{Guillemin} and b-log-symplectic manifolds \cite{Marcut}.

\subsection{Diagonal Poisson structures}
The simplest examples of Poisson structures are the Poisson structures on $\mathbb C^n$ defined by constant bivector fields
\[
\Pi= \sum_{i < j} c_{ij} \frac{\partial }{\partial x_i} \wedge \frac{\partial }{\partial x_j} \, ,
\]
where $c_{ij}$ are complex constants. Although rather particular Poisson structures, a Theorem of Darboux tells us that
any Poisson structure at the neighborhood of a point where its rank is maximal is locally analytically equivalent to
a constant Poisson structure.

The constant Poisson structures are invariant by the action of $\mathbb C^n$ on itself by translations, and therefore give rise to Poisson structures
on quotients of $\mathbb C^n$ by discrete subgroups of itself. In particular, they define Poisson structures on $(\mathbb C^*)^n$. These are defined by bivectors of the form
\[
\Pi= \sum_{i< j} c_{ij} x_i x_j \frac{\partial }{\partial x_i} \wedge \frac{\partial }{\partial x_j} \, .
\]
These bivector fields extend to bivector fields on $\mathbb P^n$, and  we will call any Poisson structure on $\mathbb P^n$ projectively equivalent to
the resulting Poisson structure,
a {\bf diagonal Poisson structure}. For generic choices of constants $c_{ij}$, the rank of the Poisson  {structure }
is $n$, when $n$ is even, or $n-1$, when $n$ is odd. In the even case, the degeneracy divisor is a
simple normal crossing divisor supported on the $n+1$ coordinate hyperplanes. Our first main result tells that the
diagonal Poisson structures are the only Poisson structures on even-dimensional Fano manifolds of dimension at least 4
with cyclic Picard group having these properties.

\begin{THM}\label{TI:A}
Let $X$ be an even-dimensional  Fano manifold of dimension at least 4 and  with cyclic Picard group.
Suppose that $\Pi$ is a generically symplectic  Poisson structure on $X$ such that
its degeneracy divisor  is a reduced  normal crossing divisor. Then $X$ is the projective
space $\lm{P}^{2n}$ and $\Pi$ is a diagonal Poisson structure.
\end{THM}

The hypothesis on the dimension of $X$ is indeed necessary, as a  Poisson structure on a smooth projective surface $S$
is nothing but a section of the anti-canonical bundle of $S$. Although $\mathbb P^2$ is the only Fano surface with cyclic Picard group,
on it any degree $3$ divisor appears as the degeneracy divisor of some Poisson structure.

One of the  key ideas in the proof of Theorem \ref{TI:A} is to show that the symplectic foliations
on the irreducible components of the degeneracy locus are defined by logarithmic $1$-forms and
then study isolated singularities of these. Under less restrictive assumptions we are
able to show that the symplectic foliation on a reduced and irreducible  {component} of the degeneracy divisor
of generically symplectic Poisson structure is defined by closed rational $1$-form, see Theorem \ref{T:closed}.

\subsection{Spaces of Poisson structures}
In dimension 2, the integrability condition for Poisson structures, $[\Pi,\Pi]=0$, is vacuous. Starting from dimension three
they impose strong constraints of $\Pi$. The study of the space of Poisson structures
on a given projective manifold $X$,
\[
\mathrm{Poisson}(X)  = \left\{ \Pi \in \mathbb P H^0(X, \bigwedge^2 TX) \, \Big| \,  [\Pi,\Pi]=0 \right\} \, ,
\]
is a challenging problem.

In dimension 3, we already know something about these spaces when $X$ is Fano with cyclic Picard group. A
Poisson structure $\Pi$ on a smooth projective $3$-fold $X$, if not zero, has rank $2$ and defines a codimension 1 foliation $
\mathcal F$ on $X$. The anti-canonical bundle
of $\mathcal F$ is effective with section vanishing on the divisorial components of the zero set of $\Pi$.
Therefore the study of Poisson structures on $3$-folds is equivalent to  the study of codimension 1 foliations with effective anti-canonical bundle.
In the case of $X=\mathbb P^3$ the description of the irreducible components of $\mathrm{Poisson}(X)$  has been carried out in \cite{CerveauLinsNetoAnnals}.
An analog description, when  $X$  is any other Fano $3$-fold with cyclic Picard group, is presented in  \cite{croco3b}.

In this paper  we prove that the diagonal Poisson structures form irreducible components
of the space of Poisson structures on $\mathbb P^n$ for any $n \ge 3$.

\begin{THM}\label{TI:B}
If $n\ge 3$ the Zariski closure  of the set of diagonal Poisson structures
in  $P H^0(\mathbb P^n, \bigwedge^2 T\mathbb P^n)$ is an
irreducible component of $\mathrm{Poisson}(\mathbb P^n)$.
\end{THM}

The proof of this result    relies on some  observations concerning the stability
of the loci where the Poisson structures has degenerate singularities in the even-dimensional case (see Subsection \ref{S:curl}); and on the stability
of codimension 1 logarithmic
foliations \cite{Omegar} in the odd-dimensional case.

\section{Poisson structures}

In this section we present the basic theory of Poisson structures on projective manifolds following
\cite{Polishchuk},  \cite{GualtieriPym} and \cite{DufourZung}.

\subsection{Basic definitions}
Let $X$ be a complex manifold.  A Poisson structure on $X$  is a $2$-derivation $\Pi\in H^0(X,\wedge^2 TX)$
such that the Schouten bracket $[\Pi,\Pi]$ vanishes identically.
If we set $\{ f , g\} := \Pi( df \wedge dg )$ then the vanishing of the Schouten bracket is equivalent to
the Jacobi identity for the Poisson bracket $\{ \cdot, \cdot \}: \wedge^2 \mathcal O_X \to \mathcal O_X$.

If $\Pi^k(p)\neq 0$ and $\Pi^{k+1}(p)=0$, then we say that $\Pi$ has rank $2k$ at $p$ and write  $\rank_p \Pi=2k$.
The biggest $2k$ such that $\Pi^k\neq0$ is the rank of the Poisson structure.

We denote by $\Pi^\sharp: \Omega^1_X\rar TX$ the $\mathcal{O}_X$-linear anchor map
defined by contraction of $1$-forms with $\Pi$.

A germ of vector field $v\in (TX)_p$ is Hamiltonian with respect to $\Pi$ if $v=\Pi^\sharp(d f)$ for some $f\in\mathcal{O}_{X,p}$. If $\Pi$
can be understood from the context, then we just say that $v$ is a Hamiltonian vector field.
A germ  $v \in (TX)_p$ is a Poisson vector field with respect to $\Pi$, if $[v,\Pi] =0$.
Note that every Hamiltonian vector field
is Poisson, but the converse does not hold. Poisson vector fields are infinitesimal symmetries of the Poisson structure and do not
need to belong to the image of $\Pi^{\sharp}$.

All the above definitions can be made on the
more general case of complex varieties, or  even schemes/analytic spaces, cf. \cite{Polishchuk} and \cite{GualtieriPym}.
For instance if $X$ is variety (reduced but perhaps singular), then we denote by $\vecs^q_X$ the sheaf of holomorphic $q$-derivations of $X$, i.e. the sheaf $\vecs^q_X = \mathrm{Hom}(\Omega^q_X,\mathcal{O}_X)$. A Poisson structure on $X$  is then a  $2$-derivation $\Pi\in H^0(X,\vecs^2_X)$  such that the Schouten bracket $[\Pi,\Pi]$ vanishes identically. The sheaf $\vecs^q_X$ coincides with $\wedge^ q TX$ over the smooth locus of $X$, but in general the natural inclusion $\wedge^q TX \to \vecs^q_X$ is strict.

\subsection{Poisson subvarieties and degeneracy loci}
Let $Y\subseteq X$ be a subvariety with defining ideal sheaf $\mathcal I_Y$. We say that $Y$ is a Poisson subvariety if $v(\mathcal{I}_Y)\subseteq\mathcal{I}_Y$ for every  Hamiltonian vector field $v$.
Equivalently, the global section $\Pi_{|Y}\in H^0(Y,(\wedge^2 TX)_{|Y})$ lies in the image of the natural map $H^0(Y, \vecs^2_Y) \to  H^0(Y,(\wedge ^2 TX)_{|Y})$.

More generally, we say that $\mathcal{I}\subseteq \mathcal{O}_X$ is a Poisson ideal if $v(\mathcal{I})\subseteq\mathcal{I}$ for every  Hamiltonian
vector field $v$. We recall that the intersection and the sum of two Poisson ideals are Poisson ideals, and that the radical of a Poisson ideal is a Poisson ideal, see \cite[Lemma 1.1]{Polishchuk}.

The $2k^{th}$ degeneracy ideal $\mathcal{I}_{2k}$ is the image of the morphism
$$\Omega^{2k+2}_X\overset{\Pi^{k+1}}\longrightarrow\mathcal{O}_X.$$
Jacobi's identity implies that $\mathcal{I}_{2k}$ is a Poisson ideal, and the  subvariety (or rather subscheme) $\D_{2k}(\Pi)$ defined by  $\mathcal I_{2k}$
is a Poisson subvariety, called the $2k^{th}$ degeneracy locus.
Note that the support of $\D_{2k}(\Pi)$ satisfies
\[
 |\D_{2k}(\Pi)|=\{p\in X; \rank_p\,\Pi\leq2k\} .
\]

If $\Pi$ is a Poisson structure of rank $2k$ then we define  the degeneracy divisor of $\Pi$, denoted by $\D(\Pi)$,
 as the divisorial component of the $(2k-2)^{\rm th}$ degeneracy locus of $\Pi$.
Note that if $\Pi$ is a generically symplectic Poisson structure
then $\D_{2k-2}(\Pi)= \D(\Pi)$, but for a general Poisson structure of rank $2k$
the degeneracy divisor $\D(\Pi)$ does not need to coincide with  $\D_{2k-2}(\Pi)$, all we have is the inclusion $\D(\Pi) \subset \D_{2k-2}(\Pi)$.

\subsection{Symplectic foliation}
Suppose that $\Pi$ is a Poisson structure of rank $2k$ on $X$. As already mentioned in the Section 1, the image of the anchor map is an involutive subsheaf of $TX$. On $U = X-\D_{2k-2}(\Pi)$, the complement
of the degeneracy locus of $\Pi$, this image is locally free and has locally free cokernel. Thus it defines a smooth foliation $\mathcal F_{|U}$. The Poisson structure induces a symplectic structure on the leaves of $\mathcal F_{|U}$. This foliation extends to a singular foliation  $\mathcal F$ on $X$
with tangent sheaf $T\mathcal F$ equal to the saturation of $\Image \Pi^\sharp$ in $TX$. The singular set of $\mathcal F$ consists of the points over which the cokernel $TX/T\mathcal F$ is not locally free and therefore has codimension at least 2.
Note that the singular set of $\mathcal F$ can be strictly smaller then the singular
set of $\Pi$.

If there are no divisorial components on $\D_{2k-2}(\Pi)$, i.e. $\D(\Pi)=0$,  then
\[
\det T\mathcal F := (\bigwedge^{rk \Pi} T\mathcal F)^{**} \simeq \mathcal O_X .
\]
In other words, the anti-canonical bundle of $\mathcal F$ is trivial. Otherwise, when there
are divisorial components in $\D_{2k-2}(\Pi)$ we get that the anti-canonical bundle of $\mathcal F$
is effective, i.e. $\det T\mathcal F \simeq \mathcal O_X(\D(\Pi))$.

The conormal sheaf of $\mathcal F$ is, by definition, the kernel of the dual of the
inclusion $T\mathcal F \to TX$. Therefore, it fits in the exact sequence
\[
0 \to N^* \mathcal F \longrightarrow \Omega^1_X \longrightarrow T\mathcal F ^* \, .
\]
The rightmost map is surjective away from the singular set of $\mathcal F$ and, since $\sing(\mathcal{F})$ has codimension at least two, we have
\[
\det N^* \mathcal F = K_X \otimes \det T\mathcal F .
\]
If we set $N= \det (N^* \mathcal F)^*$ then it follows that $\mathcal F$ is defined
by a section $\omega$ of $\Omega^q_X \otimes N$ where $q = \dim X - \dim \mathcal F$. To
wit,  $T\mathcal F$ is the kernel of the morphism
\[
TX \longrightarrow \Omega_X^{q-1} \otimes N
\]
defined by contraction with $\omega$.

If $\Pi$ defines a codimension 1 foliation on $X$, $\dim X=2n+1$, then $N^*\mathcal{F}= K_X\otimes\mathcal{O}_X(\D(\Pi))$.


\subsection{Poisson connections}
If $X$ is a manifold endowed with a Poisson structure $\Pi$ then a Poisson connection on a line-bundle $\mathcal L$ over $X$ is a
morphism of $\mathbb C$-sheaves
\[
\nabla : \mathcal L \longrightarrow  TX \otimes \mathcal L
\]
satisfying
\[
  \nabla ( f \sigma ) = f \nabla ( \sigma ) + \Pi^\sharp(df) \otimes \sigma \, ,
\]
  where $\sigma$ is any  germ  of section of $\mathcal L$ and
$f$ is any germ of function.

If $H$ is a Poisson hypersurface then the associated line-bundle  $\mathcal{L}=\mathcal{O}_X(H)$
carries a natural Poisson connection defined as follows. If $H$ is defined by $\{f_i=0\}$ with $f_i \in \mathcal O_X(U_i)$  then 
there are  trivializing sections $s_i \in H^0(U_i, \mathcal L)$  satisfying $s_i = (f_i f_j^{-1}) s_j$ over 
non-empty intersections $U_i\cap U_j$. 
The natural Poisson connection on $\mathcal L$ is defined locally by the formula
\[
\nabla s_i=-\Pi^\sharp(\frac{d f_i}{f_i})\otimes s_i \, .
\]
This is the  Polishchuk connection associated to $H$, see \cite[Section 7]{Polishchuk}.

 \subsection{Poisson structures with simple normal crossing degeneracy divisor}
We close this section on the basic theory of holomorphic Poisson manifolds by recalling a
result  by Polishchuk  \cite[Corollary 10.7]{Polishchuk} which is the starting point of our proof of Theorem \ref{TI:A}.

 \begin{thm}\label{Thm:Polishchuk}
   Let $(X,\Pi)$ be a generically symplectic  Poisson manifold with $\dim X=2n$ and $X$ be projective such that the ideal $\mathcal{I}_{2n-2}$ is reduced and the variety
   $\D(\Pi)=V_{2n-2}$ is composed by smooth irreducible components in normal crossing position. If $H^{(m)}$ consists of the points of $X$ where exactly $m$
   irreducible components of $\D_{2n-2}(\Pi)$ meet, then
   $\Pi$ has constant rank at each connected component of $H^{(m)}$. Furthermore we have $2n-2m\leq \rank\,\Pi_{|H^{(m)}}\leq2n-m$.
   \end{thm}

 In particular, $\Pi$ induces a codimension 1  foliation on each component of $\D(\Pi)$ and the singular locus of the foliation is
 contained in the intersection of, at least, two $H_i$. Moreover, let $\Pi_1=\Pi_{|H_1}$ be the Poisson structure in $H_1$ induced by $\Pi$, then
 $\D(\Pi_1)\subseteq (H_2\cup\ldots\cup H_k)\cap H_1$. If $\D(\Pi_1)$ does not contain $H_i\cap H_1$ for some $i$, then $H_i\cap H_1$ is invariant by the
  foliation induced by $\Pi_1$.

 \section{The degeneracy divisor}
In this section we will reduce the proof of Theorem \ref{TI:A} to the four-dimensional case.

\subsection{Index of a Fano manifold}
Recall a manifold $X$ is said to be Fano if its anti-canonical  bundle $KX^*$ is ample.
Assume that $X$ is Fano and has cyclic Picard group, i.e.   $\Pic\,X=\mathbb Z$.
If $H$ is an ample generator of the Picard group of $X$
then the degree of a line-bundle  $\mathcal L$ is defined by the relation
\[
\mathcal L = \mathcal O_X(\deg(\mathcal L)H) \, .
\]
The {\bf index} of $X$, denoted by $i(X)$,
is the degree of the anti-canonical divisor, i.e. $i(X) = \deg(KX^*)$.
It was proved by Kobayashi and Ochiai  \cite{KobayashiOchiai} that the index of a Fano manifold of dimension $n$ is bounded by $n+1$. Moreover,
the extremal cases are $\mathbb P^n$ ($i(X)=n+1$) and hyperquadrics $Q^n \subset \mathbb P^{n+1}$ ($i(X) =n$).

\begin{lemma}\label{Lemma:Lefschetz}
 Let $X$ be a Fano manifold with $\Pic\, X=\mathbb Z$ and $\dim X\geq 4$. Let $Y$ be a smooth hypersurface such that $\deg Y<i(X)$. Then
 $Y$ is a Fano manifold with $\Pic Y=\mathbb Z$.
\end{lemma}
\begin{proof}
Lefschetz theorem for Picard groups \cite[Example 3.1.25]{Lazarsfeld}  implies  that the restriction morphism $\Pic X \to \Pic Y$ is an isomorphism.
In particular, $\Pic Y = \mathbb Z$.

Adjunction formula gives
\[
KY = K X_{|Y} \otimes \mathcal O_Y(- Y) = \mathcal O_Y( (i(X) - \deg(Y)) Y) \, ,
\]
and our assumptions imply that $-KY$ is ample.
\end{proof}


 \subsection{The symplectic foliation on  the degeneracy divisor}

Let  $H= \sum H_i$ be a simple normal crossing divisor on a manifold $X$. A meromorphic $1$-form $\omega$ on $X$
is logarithmic with poles on $H$  if for any germ of local equation $h$ of $H$, the differentials
form $h \omega$ and $h d \omega$ are holomorphic. The sheaf of logarithmic $1$-forms with poles
on $H$, $\Omega^1_X(\log H)$ is locally free and fits into the exact sequence
$$0\rar\Omega^1_X\longrightarrow\Omega^1_X(\log H)\longrightarrow\bigoplus_{i=1}^m\mathcal{O}_{H_i}\rar0,$$
where the arrow on the right is the residue map.

\begin{prop}\label{P:log}
Let $X$ be a projective manifold of dimension $2n$ and $\Pi$ be a generically symplectic Poisson structure on it. If $\D(\Pi)$ is a
simple normal crossing divisor then for every irreducible component $Y$ of $\D(\Pi)$ the symplectic foliation on $Y$
is defined by an element of $H^0(Y, \Omega^1_Y(\log E))$ with non-zero residues on every irreducible component of  $E = (\D(\Pi)- Y)_{|Y} - \D(\Pi_{|Y})$.
\end{prop}
\begin{proof}
 Observe that $(\Pi_{|Y})^{n-1}$ is a
section of $\wedge^{2n-2} TY \simeq \Omega^1_Y \otimes KY^*$, and therefore gives rise to a twisted $1$-form
$\omega \in H^0(Y,\Omega^1_Y \otimes KY^*)$. Note that $\omega$ vanishes on $\D(\Pi_{|Y})$
and that Theorem \ref{Thm:Polishchuk} implies that the support of $\D(\Pi_{|Y})$ is contained in the support of  $(\D(\Pi)- Y)\cap Y$.  Let $s \in H^0(Y, KY^*)$
be a section  vanishing on $( \D(\Pi) - Y)\cap Y$. The quotient $\frac{\omega}{s}$ is a rational $1$-form on $Y$
with simple poles on the irreducible components of $(\D(\Pi) -Y)\cap Y  - \D(\Pi_{|Y})$.
But since these irreducible components are Poisson subvarieties of $\Pi_{|Y}$, the symplectic foliation of  {$\Pi$} is tangent
to them. It follows that  $\frac{\omega}{s}$ is logarithmic with polar set equal to  $E$.
\end{proof}


\subsection{The symplectic foliation on  the degeneracy divisor revisited }
Proposition \ref{P:log} is a particular case of the more general result below.
Although not strictly necessary for the proof of Theorem \ref{TI:A} we believe that it has some independent interest.
Under the assumption that the degeneracy divisor is smooth and irreducible, the 
result below is essentially equivalent to \cite[Proposition 1.8]{GualtieriLi}. 
The interest of the proof below is that it does not make any extra assumption on the degeneracy divisor besides reducedness.

\begin{thm}\label{T:closed}
Let $X$ be a projective manifold of dimension $2n$ and $\Pi$ be a generically symplectic Poisson structure on it.
Let $Y$ be a reduced irreducible component of   $\D(\Pi)$  and let $Y^*= Y - \sing(Y)$ be its smooth locus. Then the symplectic
foliation $\mathcal F_Y$ on $Y$  has codimension 1 and its restriction to $Y^*$ is defined by a closed meromorphic
$1$-form $\omega$ on $Y^*$  without divisorial components in its zero set, and with polar divisor satisfying
\[
  (\omega)_{\infty} =  (\D(\Pi)- Y)\cap Y^* - \D(\Pi_{|Y^*}) \, .
\]
\end{thm}
\begin{proof}
Let $p \in Y^*$ be a smooth point of $Y$. At a sufficiently small neighborhood of $p$ we can choose
local analytic coordinates $(x_1, \ldots, x_{2n-1},y)$ such that $Y=\{y=0\}$ and
\[
\Pi = y \frac{\partial}{\partial y} \wedge \left( \sum_i y^i v_i  \right)+ \left(\sum y^i \sigma_i \right)
\]
where $v_i$ are vector fields in $\{ y=0\}$, and $\sigma_i$ are a bivector fields in $\{ y=0\}$.
Since $Y$ is a reduced irreducible component of $\D(\Pi)$  and
\[
\Pi^n = y \frac{\partial} {\partial y} \wedge v_0 \wedge \sigma_0^{n-1} + y^2 \Theta
\]
for some holomorphic $2n$-vector field $\Theta$,  it follows that
$v_0 \wedge \sigma_0^{n-1} \neq 0$. But $\sigma_0^{n-1}$ defines the symplectic foliation
$\mathcal F_Y$ on $Y$, hence $\mathcal F_Y$ has codimension 1.

The integrability condition $[\Pi,\Pi]$ implies the identity $[v_0,\sigma_0]=0$. Therefore,
$v_0$ is an infinitesimal symmetry of $\mathcal F_Y$. Moreover,  $v_0 \wedge \sigma_0^{n-1} \neq 0$
implies that $v_0$ is generically transverse to $\mathcal F_Y$.
Hence, if $\eta$ is  any $1$-form
defining $\mathcal F_Y$ at a neighborhood of $p$ then the $1$-form $\frac{\eta}{\eta(v_0)}$
is closed, see, for instance, \cite[Corollary 2]{Percy}.
Note that  $\frac{\eta}{\eta(v_0)}$
has no divisorial components in its zero set, and its polar divisor is equal to
$\{\sigma_0^{n-1} \wedge v_0 = 0\} - \{ \sigma_0^{n-1}=0\}$.
Since
\[
\left\{ \left(y^{-1} {\Pi^n}\right) _{| \{ y=0\} } =0 \right\}  = \{ \sigma_0^{n-1} \wedge v_0 =0  \}  \, ,
\]
we see that the description of the poles of local closed meromorphic $1$-form $\frac{\eta}{\eta(v_0)}$ is in accordance
with the description of the poles of the sought global closed meromorphic $1$-form $\omega$.

Now take a covering of a neighborhood of  $Y^*$ by sufficiently small open subsets $U_i$ of $X - \sing(Y)$. Let $f_i$
be local equations for $Y$, satisfying $f_i = f_{ij} f_j$ for some $f_{ij} \in \mathcal O^*(U_i\cap U_j)$. Then
\[
\Pi^{\sharp} \left(\frac{df_i}{f_i} \right) - \Pi^{\sharp} \left(\frac{df_j}{f_j} \right) = \Pi^{\sharp}\left( \frac{d f_{ij}}{f_{ij}} \right)\, .
\]
The right-hand side is a Hamiltonian vector field, and in particular it is tangent to $Y$ and its restriction to $Y$
is tangent to $\mathcal F_Y$. The summands in the left-hand side are Poisson vector fields, also tangent to $Y$, but
their restrictions to $Y$ are equal to generically transverse infinitesimal symmetries of $\mathcal F_Y$. Therefore, if $\mathcal F_{Y}$
is defined by a collection of $1$-forms $\omega_i \in \Omega^1_{Y^*}(U_i)$ then over the open sets $Y^* \cap U_i\cap U_j$ we have the equality
\[
\frac{\omega_i}{i_{\Pi} (\frac{df_i}{f_i} \wedge \omega_i)}= \frac{\omega_j}{i_{\Pi} (\frac{df_j}{f_j} \wedge \omega_j )} \, .
\]
As argued before, these $1$-forms are closed and therefore patch together to give the sought closed meromorphic $1$-form defining
$\mathcal F_Y$ on $Y^*$.
\end{proof}

\subsection{Irreducible components of the degeneracy divisor}
After the brief digression about the symplectic foliation on general reduced and irreducible components of
the degeneracy divisor of generically symplectic Poisson structures  we  return to the proof of Theorem \ref{TI:A}.

\begin{lemma}\label{Lemma:NF=OX}
If $X$ is a projective manifold with $\Pic X= \mathbb Z$ then there is no foliation on $X$ with trivial normal bundle.
Moreover, there is no smooth codimension 1 foliation $\mathcal{F}$ on $X$.
\end{lemma}
\begin{proof}
If the normal bundle of $\mathcal F$ is trivial then  $\mathcal F$ is defined
by a non-zero section of $\Omega^1_X$. But according to the  Hodge decomposition
we have that  $H^0(X,\Omega^1_X)\simeq H^1(X,\mathcal{O}_X)$, and as  $\Pic X = \mathbb Z$, the latter group is zero. This proves the first part of the statement.

For the second part, note that if $\mathcal F$ has no singularities then Baum-Bott formula
implies that $c_1(N\mathcal F)^{\dim X}=0$. Since $\Pic X = \mathbb Z$, it follows that $N \mathcal F= \mathcal O_X$ and we can conclude as before.
\end{proof}

\begin{prop}\label{Prop:Singularset}
Let $X$ be a Fano manifold of dimension $2n\ge 4$ with cyclic Picard group.
Let $\Pi\in H^0(X,\bigwedge^2 TX)$ be a generically symplectic Poisson structure.
 Assume that $\D(\Pi)$ is a reduced simple normal crossing divisor.
 If $Y$ is an irreducible component of $\D(\Pi)$ and $\Pi_{|Y}$ is the induced Poisson structure on $Y$  then
$\D (\Pi_{|Y})\subsetneqq( \D(\Pi) - Y )\cap Y$. Moreover, there exists another irreducible component $Z$ of $\D(\Pi)$ such that   the induced Poisson structure
 on $Y\cap Z$ is generically symplectic.
\end{prop}
\begin{proof}
Theorem \ref{Thm:Polishchuk} implies $\D(\Pi_{|Y})\subseteq(\D(\Pi) - Y)\cap Y$.
We want to  prove that the inclusion is strict.
If the inclusion is not strict then Proposition \ref{P:log} implies that $\mathcal F_Y$
is defined by a global holomorphic $1$-form. But this contradicts  Lemma \ref{Lemma:NF=OX},
proving that the inclusion must be strict. In particular, the generic rank of $\Pi_{|Y}$ is equal 
to $2n-2$.

If we take an irreducible component $Z$ of $\D(\Pi)$ such that $Y\cap Z$ is not contained
in $\D(\Pi_{|Y})$ then the induced Poisson structure on $Y\cap Z$ has rank $2n-2 = \dim (Y \cap Z)$  
and is therefore generically symplectic on $Y\cap Z$.
 \end{proof}

\begin{prop}\label{P:k>=3}
Let $X$ be a Fano manifold of dimension $2n\ge 4$ with cyclic Picard group.
Let $\Pi\in H^0(X,\bigwedge^2 TX)$ be a generically symplectic Poisson structure.
If $\D(\Pi)$ is a reduced simple normal crossing divisor then $\D(\Pi)$ has at least 3
distinct irreducible components. Moreover, if $\D(\Pi)$ has exactly three irreducible components
then for any irreducible component $Y$ of $\D(\Pi)$,  the symplectic  foliation $\mathcal F_Y$ on $Y$ has normal
bundle equal to $KY^*$ and trivial {canonical} bundle.
\end{prop}
\begin{proof}
Let $Y$ be an irreducible component of $\D(\Pi)$. The symplectic foliation on $Y$ is defined by a logarithmic $1$-form
$\omega$. The residue of $\omega$ is a $\mathbb C$-divisor with zero Chern class. Therefore the polar set of $\omega$ has
at least two distinct irreducible components. As $\Pic Y = \mathbb Z$,  Lemma \ref{Lemma:Lefschetz}, any other irreducible component
$Z$ of the simple normal crossing divisor $\D(\Pi)$ intersects $Y$ along an smooth and irreducible hypersurface. As $(\omega)_{\infty}$ is
contained in $(\D(\Pi)-Y )\cap Y)$ it follows that $\D(\Pi)$ must have at least two other irreducible components besides $Y$.

The symplectic foliation on $Y$ is defined by a logarithmic $1$-form with
polar set equal to $(\D(\Pi) - Y) \cap Y - \D(\Pi_{|Y})$, and normal bundle given by
the associated line-bundle. If we have only three irreducible components in $\D(\Pi)$   then
$\D(\Pi_{|Y})$  is the zero divisor as there are no holomorphic $1$-forms on $Y$, and any  logarithmic $1$-form
must have at least two irreducible components in its polar divisor by the residue theorem.
Therefore $N\mathcal F_{Y} = KY^* $ and $K \mathcal F_Y = \mathcal O_Y$.
\end{proof}

\begin{cor}\label{C:k>=3}
If  a Fano manifold of dimension $2n\ge 4$ with cyclic Picard group admits
a generically symplectic Poisson structure with simple normal crossing degeneracy divisor then the index of $X$ is at least 3.
\end{cor}
\begin{proof}
Since  $KX^*  = \mathcal O_X(\D(\Pi))$, the index of $X$ is the sum of the degrees of the irreducible components of $\D(\Pi)$.
Proposition \ref{P:k>=3} implies
$i(X)\ge 3$ as wanted.
\end{proof}

\subsection{Induction argument} Recall from the Section 1 that a diagonal Poisson structure on $\mathbb P^n$ is
defined on a suitable affine chart by a bivector field of the form $\sum_{ij} \lambda_{ij} x_i x_j \frac{\partial}{\partial x_i} \wedge \frac{\partial}{\partial x_j}$.

\begin{prop}\label{P:concluir}
Let $\Pi$ be a generically symplectic  Poisson structure on $\mathbb P^{2n}$. If $\D(\Pi)$ is the union
of $2n+1$ hyperplanes in general position then $\Pi$ is projectively equivalent to a diagonal Poisson structure.
\end{prop}
\begin{proof}
A Poisson structure $\Pi$ on $\mathbb P^{2n}$ is determined by a Poisson structure $\tilde \Pi$ on $\mathbb C^{2n+1}$ defined by
a homogenous quadratic bivector field, see for instance \cite[Theorem 12.1]{Polishchuk} and  \cite{Bondal}. Since the hyperplanes in $\D(\Pi)$ are in general position, we can choose homogeneous coordinates
such that $\D(\Pi)= \{ x_0 \cdots x_{2n} = 0 \}$.

If we write
\[
\tilde \Pi = \sum_{i\le j;k < l} a_{ij}^{kl} x_i x_j \frac{\partial}{\partial x_k} \wedge \frac{\partial}{\partial x_l}
\]
then the Hamiltonian vector field induced by $x_m$  is
\[
\sum_{i\le j; m< l}   a_{ij}^{ml}  x_i x_j \frac{\partial}{\partial x_l}  - \sum_{i\le j; k< m}   a_{ij}^{km}  x_i x_j \frac{\partial}{\partial x_k} \, .
\]
For every  $m'$, the hypersurface $\{ x_{m'} =0\}$ is invariant by this vector field. Thus, if $m<m'$ then  $a_{ij}^{mm'}x_i x_j$ must be  divisible by $x_{m'}$. Changing the
roles of $m$ and $m'$, we  conclude that $a_{ij}^{mm'}=0$ unless $\{ i,j\} = \{ m, m'\}$. It follows that $\tilde \Pi$ is diagonal and so is $\Pi$.
\end{proof}

\begin{prop}\label{P:inducao}
 If Theorem \ref{TI:A} holds for Fano manifolds   {with cyclic Picard group} of dimension $2n$, with $n\geq2$, then it also holds for Fano manifolds   {with cyclic Picard group} of dimension  $2(n+1)$.
\end{prop}
\begin{proof}
Let $X$ be a Fano manifold  {with cyclic Picard group} of dimension $2(n+1)$ with a generically symplectic Poisson structure $\Pi$ having a reduced simple normal crossing degeneracy divisor.
 Let $Y$ and $Z$ be irreducible components of $\D(\Pi)$ as Proposition \ref{Prop:Singularset}. Then  $W = Y\cap Z$ is a Fano manifold {with cyclic Picard group} of dimension $2n$ and the Poisson structure $\Pi_{|W}$ is also generically symplectic. Moreover,  $D(\Pi_{|W})=(D(\Pi)-Y-Z)\cap W$ and consequently  $D(\Pi_{|W})$ is
 a reduced simple normal crossing divisor.
According to  our assumptions, $W=\lm{P}^{2n}$ and the adjunction formula says that  $i(X)=\deg Y+\deg Z+i(W)$.
Kobayashi-Ochiai Theorem implies that $\deg Y=\deg Z=1$ and $i(X)=2n+3$, i.e., $X=\lm{P}^{2n+2}$.
Since the degeneracy divisor of  $\Pi_{|W}$ is the union of  $2n+1$ hyperplanes in general position, we conclude that the singular set of $(X,\Pi)$
 consists of $2n+3$ hyperplanes in general position.
 \end{proof}

 \section{Poisson structures on Fano 4-folds}

Proposition \ref{P:inducao} reduces the proof of  Theorem \ref{TI:A} to the four-dimensional case which will be carried out in this section.

\subsection{Existence of  global vector field and constraints on the index} We will prove below, Proposition \ref{Prop:i(X)>3}, that to prove Theorem \ref{TI:A} in dimension 4 we
can assume that $X$ is $\mathbb P^n$ or an hyperquadric.

\begin{lemma}\label{Lemma:vectorfield}
Let $\Pi$ be a generically symplectic Poisson structure on a complex manifold $X$, with $\dim X\ge 4$.
Let $Y$ and $Z$ be two distinct effective  divisors which are
linearly equivalent. If $Y$ and $Z$ are Poisson divisors then there exists
a non-zero  global vector field $v \in H^0(X,TX)$ which is
a Hamiltonian vector field in $X\setminus(Y\cup Z)$.
\end{lemma}
\begin{proof}
Since $\Pi$ is a  generically symplectic  Poisson structure,  the anchor morphism $\Pi^\sharp:\Omega^1_X\rar TX$ is injective.

Write $\mathcal{L}=\mathcal{O}_X(Y)=\mathcal{O}_X(Z)$. Consider the Polishchuk connections $\nabla_Y$ and $\nabla_Z$ associated to
$Y$ and $Z$ respectively. Then $\nabla_Y-\nabla_Z:\mathcal{L}\rar TX\otimes\mathcal{L}$ is a $\mathcal{O}_X$-linear map and
so it induces a global vector field $v\in H^0(X,TX)$.

If $Y$ and $Z$ are locally defined by $\{f=0\}$ and $\{g=0\}$, respectively,
then $v=\Pi^\sharp(d (\log\frac{f}{g}))$. Since $Y$ and $Z$ are distinct,  we have
$d (\log\frac{f}{g})\neq0$  and, as $\Pi$ is generically symplectic,  $v$ does not vanish identically.
From the local expression, we see that $v$ is Hamiltonian vector field
in $X\setminus(Y \cup Z)$.
\end{proof}

\begin{prop}\label{Prop:i(X)>3}
Let $X$ be a Fano manifold of dimension 4 with cyclic  Picard group.
If there exists a generically symplectic Poisson structure on $X$ with
simple normal crossing {degeneracy} divisor then the index of $X$ is at least 4. In
particular, $X$ is a four-dimensional hyperquadric or $\mathbb P^4$.
\end{prop}
\begin{proof}
As we  already know that $i(X) \ge 3$, see Corollary \ref{C:k>=3}, we can assume  $i(X)=3$.
Let $\Pi$ be a Poisson structure on $X$ satisfying the assumptions, and $Y$ be
an irreducible component of $\D(\Pi)$. Note that $\D(\Pi)$ has exactly three
irreducible components (say $Y,Z,$ and  $W$),
each one of them has degree 1, and  any two of them are linearly equivalent.

By Lemma \ref{Lemma:vectorfield}, we have a global vector field $v \in H^0(X,TX)$ induced
by $Z$ and $W$, which is tangent to $Y$. Wahl's theorem \cite{Wahl} ensures that $v$ does not
vanish identically along $Y$, and therefore we have a non-zero vector $v_Y \in H^0(Y,TY)$.
Since $v$ is a Hamiltonian vector field in $ X - ( Z \cup W)$, the vector field
$v_Y$ is tangent to the symplectic  foliation $\mathcal{F}_{Y}$ on $Y$.

Adjunction formula implies that $Y$ is a Fano $3$-fold of index two, and we have just proved that
$Y$  carries a foliation $\mathcal F_Y$ with trivial canonical bundle (Proposition \ref{P:k>=3}) which  {satisfies} $h^0(Y, T \mathcal F_Y)>0$.
This suffices to characterize $Y$ and $\mathcal F_Y$. On the one hand, \cite[{Corollary 7.4}]{croco3b}
implies that $Y$ is isomorphic to $X_5$, the unique Fano $3$-fold with Picard group generated by an element $H$ satisfying $H^3 = 5$.
On the other hand, \cite[Lemma 4.2 and Theorem 7.1]{croco3b} imply that the  foliation $\mathcal F_Y$ has trivial tangent bundle  and is induced by an action
of $\Aff(\mathbb C)$ in $X_5$.

One of the first steps of the  proof of \cite[Theorem 7.1]{croco3b} is to show that the foliation on $X_5$
induced by the action of $\Aff(\mathbb C)$ is not defined by a logarithmic $1$-form
with poles on two hyperplanes sections.  {This contradicts Proposition \ref{P:log}} and shows that $i(X) \ge 4$.
\end{proof}

\subsection{Singularities of logarithmic foliations}
Let $X$ be a four-dimensional Fano manifold with cyclic Picard group and let $\Pi$ be
a Poisson structure on $X$ with simple normal crossing degeneracy divisor.

Fix an irreducible component  $Y$ of $\D(\Pi)$. We want to analyze the singularities
of the symplectic  foliation $\mathcal F_{Y}$ on $Y$. Recall that  $\mathcal F_{Y}$
is defined by a logarithmic $1$-form $\omega \in H^0(Y,\Omega^1_Y(\log E))$ where
$E= (\D(\Pi) - Y ) \cap Y - \D(\Pi_{|Y})$. If $p$ is a point at the intersection of
$m$ irreducible components of $E$ then
\[
\omega = \sum_{i=1}^m \lambda_i \frac{dx_i}{x_i} + \beta \,
\]
where $x_1, \ldots, x_m$ are defining functions for the irreducible components of $E$ through $p$,
$\lambda_1, \ldots, \lambda_m$ are nonzero complex numbers and  $\beta$ is a holomorphic $1$-form.
It follows that $\mathcal F_Y$ does not have isolated singularities at a neighborhood of $E$. Instead
the singular set of $E$ coincides with the singular of $\mathcal F_Y$ at a neighborhood of $E$.
Since $E$ is an ample  normal crossing divisor
on $Y$ then components of  singular set of $\mathcal F_Y$ disjoint from $E$ must be zero dimensional,
for details see \cite{CukSoaVains}. The argument above also shows that
$\omega$, seen as a section of $\Omega^1_Y(\log E)$, does not have zeros at a neighborhood of $E$. In particular
the number of isolated singularities of $\mathcal F_Y$, counted with multiplicities, is equal to
the top Chern class of $\Omega^1_Y(\log E)$.

\begin{lemma}\label{Lem:Technicallemma}
 Let $\Pi$ be a generically symplectic Poisson structure on a Fano 4-fold $X$ with simple normal crossing
 degeneracy divisor.
If $Y$ is an irreducible component of $\D(\Pi)$ and $\mathcal F_Y$ is the  foliation on $Y$ induced by $\Pi$ then
any isolated singularity $p$ of $\mathcal F_Y$ lies at the intersection of at least 3 distinct irreducible
components of $\D(\Pi)$.
\end{lemma}
\begin{proof}
 If $p$ is a singular point of $\mathcal{F}_Y$, then $\Pi(p)=0$. Locally, this means that
 $\Pi \in \mathfrak{m}_p\otimes \wedge^2 TX$. In particular,
 $\Pi\wedge\Pi\in\mathfrak{m}^2_p\otimes\wedge^4 TX$. Since $\D(\Pi)$ is normal crossing, we can
 find local coordinates $(x_1,x_2,x_3,x_4)$ in a neighborhood of $p=0$, such that
 $\Pi\wedge\Pi=x_1x_2V$, where $V$ is a $4$-derivation.
 Write $\Pi=\Pi_1+\Pi_2+\ldots$ the Taylor series of $\Pi$. To prove the lemma, we just need to check
 that $\Pi_1\wedge\Pi_1=0$. Since $\Pi_1$ is a linear Poisson structure, it can be reinterpreted as
 a Lie algebra on $(\lm{C}^4)^*$ (see \cite{DufourZung}, Chapter 1).
 We have a complete classification of the Lie algebra structure in dimension 4 and
 a simple check of the table in \cite[Lemma 3]{BurdeSteinhoff} shows that $\Pi_1\wedge\Pi_1\neq0$ just in the cases
 $\mathfrak{aff}(\lm{C})\times\mathfrak{aff}(\lm{C})$, $\lieg_6$ and $\lieg_8(\alpha)$.

 The last two cases are excluded because for them we have $\Pi_1\wedge\Pi_1=x_4^2\dbyd{}{x_1}\wedge\ldots\wedge\dbyd{}{x_4}$ which is not coherent with the assumption that $\D(\Pi)$ is a simple normal crossing divisor.
 To exclude the first case,
 we use the theorem of Dufour and Molinier  \cite[{Theorem 4.4.12}]{DufourZung} which states that we can find coordinates $(y_1,\ldots,y_4)$,
 $Y=\{y_1=0\}$ such that
 $$\Pi=y_1\dbyd{}{y_1}\wedge\dbyd{}{y_2}+y_3\dbyd{}{y_3}\wedge\dbyd{}{y_4}.$$
{The  foliation induced by $\Pi$ on $Y$ is given by the kernel of the holomorphic $1$-form $\mathrm{d}y_2$ and, so,   is regular at $p$.}
 \end{proof}

\begin{lemma}\label{L:Q3}
Let $\omega$ be a logarithmic $1$-form with simple normal crossing polar divisor $D$  on the quadric $3$-fold $Q^3$.
If the degree of $D$ is at most three then $\omega$ admits an isolated singularity.
\end{lemma}
\begin{proof}
First assume that $D$ has degree 2, i.e. $D$ is the union of two hyperplane sections $H_1$ and $H_2$ intersecting $Q^3$ transversely.
Note that $h^0(Q^3,\Omega^1_{Q^3}(\log H_1 +H_2))=1$, and the foliation induced by any section of
$\Omega^1_{Q^3}(\log H_1 +H_2)$ is the pencil of hyperplane sections generated by $H_1$ and $H_2$.
Since the dual variety of $Q^3$ is also a quadric, in the pencil generated by $H_1$ and $H_2$ there are two elements
which intersect $Q^3$ on a cone over a two-dimension smooth quadric $Q^2$. Therefore, any nonzero $\omega_0 \in H^0(Q^3,\Omega^1_{Q^3}(\log H_1 + H_2))$
has two isolated singularities (counted with multiplicities).

Suppose now that $D$ is the union of three hyperplane sections intersecting transversely, say $H_1, H_2,$ and $H_3$. Through the inclusion
$\Omega^1_{Q^3}(\log H_1 + H_2) \to \Omega^1_{Q^3}(\log H_1 + H_2 + H_3)$ we can interpret $\omega_0 \in H^0(Q^3,\Omega^1_{Q^3}(\log H_1 + H_2))$
as a section of $\Omega^1_{Q^3}(\log H_1 + H_2 + H_3)$, and as such it vanishes not only at  the zeros of the corresponding rational $1$-form
but also at $H_3$. Nevertheless, sufficiently small general perturbations of  $\omega_0$ inside $H^0(Q^3,\Omega^1_{Q^3}(\log H_1 + H_2+ H_3))$
will still have isolated singularities near the original isolated singularity of $\omega_0$. This suffices to show that
$c_3(\Omega^1_{Q^3}(\log H_1 + H_2 + H_3) ) \ge c_3(\Omega^1_{Q^3}(\log H_1 + H_2 ) ) \ge 2$.

The remaining case, $D$ is the union of a smooth hyperplane section and a smooth hypersurface of degree 2 intersecting transversely,
can be dealt with similarly. Alternatively, a straightforward computation shows that $c_3(\Omega^1_{Q^3}(\log D))=8$.
\end{proof}

\begin{remark}
The constraints on the degree of $D$ and on the dimension of hyperquadric $Q$ are not really necessary.
The {\it continuity} argument used above can be pushed to
prove that any logarithmic $1$-form on $Q^{n}$, $n\ge 3$, with simple normal crossing polar divisor has isolated singularities.
\end{remark}

\subsection{Proof of Theorem \ref{TI:A}} The result below together  with Propositions \ref{P:concluir} and \ref{P:inducao}
clearly imply Theorem \ref{TI:A}.

\begin{prop}
Let $X$ be a  {Fano} $4$-fold with cyclic Picard group. If there exists a generically symplectic Poisson
structure on $X$ with simple normal crossing degeneracy divisor then $X$ is $\mathbb P^4$ and $\D(\Pi)$ is the
union of five hyperplanes in general position.
\end{prop}
\begin{proof}
Proposition \ref{Prop:i(X)>3} implies $X$ is a four dimension hyperquadric $Q^4$ or $\mathbb P^4$. Aiming at
a contradiction let us assume $X = Q^4$. Since the index of $Q$ is $4$, we have that $\D(\Pi)$ has three
or four irreducible components with degrees summing up to $4$. Let $Y$ be an irreducible component of degree 1.
Thus $Y = Q^3$ is a three-dimensional hyperquadric. The symplectic foliation on $Q^3$ is defined
by logarithmic $1$-form with at least two hypersurfaces in its polar set. According to Lemmas  \ref{Lem:Technicallemma} and \ref{L:Q3}, $\mathcal F_Y$ has isolated singularities which
lies at the intersection of at least 3 distinct irreducible components of $\D(\Pi)$. Thus, it must lie also on the polar divisor
of $\omega$. But, as already pointed out $\mathcal F_Y$ does not have isolated singularities at a neighborhood of the polar divisor of $\omega$. This contradiction
proves that $X= \mathbb P^4$.

Assume now that $X=\mathbb P^4$. If the conclusion does not hold then $\D(\Pi)$ contains a hyperquadric $Q$ or a cubic hypersurface $C$. When $\D(\Pi)$
contains a hyperquadric then the argument of the previous paragraph leads to a contradiction.  If instead $\D(\Pi)$ contains a cubic hypersurface $C$ but does no contain
a hyperquadric 
then $\D(\Pi) = C + H_1 + H_2$  for suitable hyperplanes $H_1, H_2$. The induced foliation
 $\mathcal F_C$ on $C$ is defined by a logarithmic $1$-form with poles on $H_1\cap C$ and $H_2\cap C$ according to Proposition \ref{P:log}. Therefore 
 the leaves of $\mathcal F_C$ are elements of a pencil of hyperplane sections of $C$. The singular members of this pencil (which exist because the dual
of a smooth cubic is an hypersurface)  would give isolated singularities
for $\mathcal F_C$ disjoint of $\D(\Pi)$. In both cases we arrive at contradictions which imply that $D(\Pi)$ is a union of five hyperplanes as claimed.
\end{proof}

\section{Stability of diagonal Poisson structures}

\subsection{Curl operator}\label{S:curl}
Let us fix a neighborhood $U$ of $0 \in \mathbb C^{n}$ and a nowhere vanishing $n$-form $\Omega$
on $U$, e.g. $\Omega=d x_1\wedge \ldots\wedge d x_n$.
For every $p=0,1,\ldots,n$, the map
$$\Omega:\bigwedge^p T {U}\rar\Omega^{n-p}_{U}$$
defined by $\Omega(A)=i_A(\Omega)$, is an $\mathcal{O}_{U}$-linear isomorphism from the space $\wedge^p TU$ to $\Omega^{n-p}_{\lm{C}^{n}}$.
The inverse map will be denoted
as $\Omega^{-1}:\Omega^{n-p}_{U}\rar\bigwedge^p T {U}$.

The linear operator defined by the composition $\Omega^{-1}\circ d \circ \Omega$ is called the curl operator and it is denoted by $D_\Omega$.

If $\Pi$ is a Poisson structure on $U$  then the vector field $D_\Omega \Pi$ is called the curl vector field (with respect to $\Omega$) of $\Pi$.
As suggested by the notation, the curl vector field does depend on the choice of $\Omega$. But the set of points where both $\Pi$ and $D_{\Omega} \Pi$
vanish, is independent of the choice of a nowhere vanishing $n$-form. We will call this set, the set of degenerate singular points of $\Pi$.

According to  \cite[{Lemma 2.6.9}]{DufourZung}  the identity
\[
 [D_\Omega\Pi,\Pi]=0
\]
holds true, i.e. $D_\Omega \Pi$ is a Poisson vector field for $\Pi$.

\begin{lemma}\label{L:sss}
Assume $n = 2m \ge 4$. Let $\Pi= \sum_{i<j} \lambda_{ij} x_i x_j \frac{\partial}{\partial x_i} \wedge \frac{\partial}{\partial x_j}$ be a
general diagonal Poisson structure on $U$, i.e. the complex numbers $\lambda_{ij}$ are general.
Let $T$ be an irreducible complex variety containing a point $t_0$ and let
$\Pi_{t}$, $t \in T$, be a holomorphic family of Poisson structures
on $U$ such that $\Pi_{t_0} = \Pi$. Then, after   restricting $U$ and $T$,
 there exists $\gamma : T \to U$ such that $\gamma(t)$ is the unique degenerate singular point of $\Pi_t$ in $U$. 
 Moreover,  the vanishing order of $(\Pi_t)^m$ at $\gamma(t)$ is at least $n=2m$.
\end{lemma}
\begin{proof}
Let $\Omega= dx_1 \wedge \cdots \wedge dx_n$ be the standard volume form on $\mathbb C^n$.
The curl of $\Pi$ with respect to $\Omega$ is
\[
D_{\Omega} \Pi = \sum_{i} \mu_i  x_i \frac{\partial}{\partial x_i}, \quad \text{ where } \quad \mu_i = \sum_{i<j} \lambda_{ij} - \sum_{i>j}   \lambda_{ji} \, .
\]
For a general choice of $\lambda_{ij}$ the origin is the unique singularity of $D_{\Omega} \Pi$. This singularity is simple in the sense that
the ideal generated by the coefficients of $D_{\Omega} \Pi$ coincides with the maximal ideal of  $\mathcal O_{U,0}$.  Therefore
 for sufficiently small $t$, $D_\Omega\Pi_t$ has a unique simple singularity $\gamma(t)$ close to the origin. Implicit function theorem
 implies that $\gamma : T \to U$ is a holomorphic function.

Note that $\sum_{i} \mu_i=0$ and  that the linear map
\[
\sum_{ij} \lambda_{ij} x_i x_j \frac{\partial}{\partial x_i} \wedge\frac{\partial}{\partial x_j }  \longmapsto \sum_i x_i \mu_i
\]
has rank $n-1$. In particular, if the complex numbers $\lambda_{ij}$  are sufficiently general then
$
\sum_{i} c_i \mu_i =0
$
with $c_i \in \mathbb Z$ implies that $c_1= \cdots = c_n$. For $t$ sufficiently general, the eigenvalues of the linear
part of $D_{\Omega} \Pi_t$ at $\gamma(t)$ will have the same property.
According to \cite{MartinetBourbaki}, we have a formal change of coordinates  centered at $\gamma(t)$,
which transforms $D_{\Omega} \Pi_t$ into a  vector field
\[
   v_t= D_{\Omega} \Pi_t = \sum_{i} A_{i,t}(x_1 \cdots x_n)  x_i \frac{\partial}{\partial x_i} \,  \, ,
\]
where $A_{i,t}$ are germs holomorphic functions in $(\mathbb C,0)$ satisfying $A_{i,t}(0)= \mu_{i,t} \neq 0$.

Since $v_t$ is Poisson vector field for $\Pi_t$ we have that $[v_t,\Pi_t]=0$. If we write
$\Pi_t = \sum \pi_{ij,t} \frac{\partial}{\partial x_i} \wedge \frac{\partial}{\partial x_j}$ then  we
obtain
\[
   \sum_{i,j} \pi_{ij,t}(0) ( \mu_{i,t} + \mu_{j,t} ) {\dbyd{}{x_i}\wedge\dbyd{}{x_j}} = 0 \, .
\]
Therefore for every $i,j$ the identity  $\pi_{ij,t}(0) = 0 $ holds true since $\mu_{i,t}$ and $\mu_{j,t}$
are $\mathbb Z$-linearly independent. Thus $\Pi_t$ vanishes at zero, and consequently  $\gamma(t)$ is a degenerate singular point for $\Pi_t$. 

Looking at the
linear part of $\Pi$ at zero we obtain that
\[ {
 (\mu_{i,t} + \mu_{j,t} - \mu_{k,t})\dbyd{\pi_{ij,t}}{x_k}(0)=0 \text{ for every $i$, $j$ and $k$.}}
 \]
Again by the $\mathbb Z$-linear independence of $\mu_{i,t},  \mu_{j,t}$ and $\mu_{k,t}$, we deduce that the linear part of $\Pi_t$
also vanishes at zero. Therefore $(\Pi_{t})^m$ vanishes at zero with order greater than or equal to  $2m$.
\end{proof}

\begin{lemma}\label{L:hyper}
 Let $H$ be a reduced hypersurface in $\mathbb P^k$ of degree $k+1$. Suppose that $H$ has $k+1$ singular points  in general position.
 If the algebraic multiplicity of each of these $k+1$ points is $k$ then  $H$ is the union of $k+1$ hyperplanes in general position.
\end{lemma}
\begin{proof}
Let $p_0, \ldots, p_k$ be the $k+1$ points of $H$ with algebraic multiplicity $k$. We can assume that
$p_0=[1:0: \ldots :0], p_1 = [0:1: \ldots:0], \ldots, p_k=[0:\ldots:0:1]$. Let $f \in \mathbb C[x_0, \ldots, x_n]$ be a homogenous
polynomial of degree $k+1$ cutting out $H$. Since $H$ has algebraic multiplicity $k$ at $p_i$, it follows that the polynomials
$\frac{\partial^2 f}{\partial x_i^2}$
vanish identically for every $i\in \{ 0, \ldots, k\}$. In other words, every monomial contributing to the Taylor expansion of $f$ at $0 \in \mathbb C^{k+1}$ is square-free. But there is
only one square-free monomial of degree $k+1$, $x_0 \cdots x_{k+1}$. The lemma follows.
\end{proof}

\subsection{Proof of Theorem \ref{TI:B}} Let us recall the statement of Theorem \ref{TI:B}.

\begin{thm}\label{thm:B}
If we take sufficiently small deformations of
a generic diagonal Poisson structure in $\lm{P}^{n}$  then
the resulting Poisson structures are still diagonal Poisson structures.
\end{thm}
\begin{proof}
Assume first that $n=2k+1$ is odd. If $\Pi$ is a generic Poisson structure then it has rank $2k$. The symplectic foliation
is nothing but the logarithmic foliation defined by
\[
\omega= \left(\prod_{i=0}^{2k+1 }x_i\right) \left( \sum_{i=0}^{2k+1} \lambda_i \frac{dx_i}{x_i} \right) \in H^0(\mathbb P^{2k+1}, \Omega^1_{\mathbb P^{2k+1}}(2k+2)) \,
\]
where $\lambda_i \in \mathbb C$ satisfy $\sum_{i=0}^{2k+1} \lambda_i =0$.   Moreover, any choice of complex numbers $\lambda_i$ summing up to zero defines
a codimension 1 logarithmic foliation $\mathcal F_{\omega}$ which is the symplectic foliation of a diagonal Poisson structure. To prove this note that the tangent sheaf of the foliation
is trivial and its space of global sections is vector space of dimension $2k$ of commuting vector fields tangent to the hypersurface $\{ x_0 \cdots x_{2k+1}=0\}$. Any element
of $\bigwedge^2 H^0(\mathbb P^{2k+1}, T\mathcal F_{\omega})$ having rank $2k$ defines the sought Poisson structure. The stability of a general diagonal Poisson structure
on $\mathbb P^{2k+1}$ follows from the corresponding result for the stability of codimension 1 logarithmic foliations with poles on $2k+2$ hyperplanes, see the main result of \cite{Omegar} or \cite[Example 6.2]{CukiermanJorge}.

Assume now that $n=2k$ is even. If $\Pi$ is a generic diagonal Poisson structure in $\mathbb P^{2k}$ then it is generically symplectic and $\Pi$ has $2k+1$ degenerate singular points.
Lemma \ref{L:sss} implies that any small deformation $\Pi_{\varepsilon}$ of $\Pi$ will still have  $2k+1$ degenerate singular points  and $\Pi_{\varepsilon}^k$ has vanishing
order $2k$ at each of these points.
 Consequently, the
degeneracy divisor $\D(\Pi_{\varepsilon})$ has $2k+1$ points of multiplicity $2k$ and Lemma \ref{L:hyper}  implies that it must be the union of $2k+1$ hyperplanes in general position.
Theorem \ref{TI:B}  follows from Proposition \ref{P:concluir}.\end{proof}

\bibliographystyle{amsplain}

\end{document}